\documentclass[12pt,reqno]{amsart}
\usepackage{fullpage}
\allowdisplaybreaks
\usepackage{times}
\usepackage{colonequals}
\usepackage{amsmath,amssymb,amsthm,url}
\usepackage[utf8]{inputenc}
\usepackage[english]{babel}
\usepackage[alphabetic]{amsrefs}
\usepackage{bbm}
\usepackage{enumerate}
\usepackage{bm}
\usepackage{graphicx}
\usepackage{mathrsfs}
\usepackage[colorlinks=true, pdfstartview=FitH, linkcolor=blue, citecolor=blue, urlcolor=blue]{hyperref}
\newtheorem{thm}{Theorem}[section]
\newtheorem{lem}[thm]{Lemma}
\newtheorem{prop}[thm]{Proposition}

\newtheorem{conj}[thm]{Conjecture}

\theoremstyle{definition}

\newtheorem{question}[thm]{Question}
\newtheorem{rem}[thm]{Remark}

\newtheorem{ex}[thm]{Example}

\newcommand{\F}{\mathbb{F}}
\newcommand{\Fq}{\mathbb{F}_q}

\AtBeginDocument{%
	\def\MR#1{}
}

\begin{document}

\title{Plane-filling curves of small degree over finite fields}

\author{Shamil Asgarli}
\address{Department of Mathematics and Computer Science \\ Santa Clara University \\ 500 El Camino Real \\ USA 95053}
\email{sasgarli@scu.edu}

\author{Dragos Ghioca}
\address{Department of Mathematics \\ University of British Columbia \\ 1984 Mathematics Road \\ Canada V6T 1Z2}
\email{dghioca@math.ubc.ca}


\begin{abstract}
A plane curve $C$ in $\mathbb{P}^2$ defined over $\mathbb{F}_q$ is called plane-filling if $C$ contains every $\mathbb{F}_q$-point of $\mathbb{P}^2$. Homma and Kim, building on the work of Tallini, proved that the minimum degree of a smooth plane-filling curve is $q+2$. We study smooth plane-filling curves of degree $q+3$ and higher. 
\end{abstract}

\maketitle 

\section{Introduction}

The study of space-filling curves in $\mathbb{R}^2$ starts with the work of Peano \cite{Pea90} in the 19th century. About 100 years later, Nick Katz \cite{Kat99} studied space-filling curves over finite fields and raised open questions about their existence. One version of Katz's question was the following. Given a smooth algebraic variety $X$ over a finite field $\mathbb{F}_q$, does there always exist a \emph{smooth} curve $C\subset X$ such that $C(\mathbb{F}_q)=X(\mathbb{F}_q)$? In other words, is it possible to pass through all of the (finitely many) $\mathbb{F}_q$-points of $X$ using a smooth curve? Gabber~\cite{Gab01} and Poonen \cite{Poo04} independently answered this question in the affirmative. 

We will consider the special case when $X=\mathbb{P}^2$. We say that a curve $C\subset \mathbb{P}^2$ is \emph{plane-filling} if $C(\mathbb{F}_q)=\mathbb{P}^2(\mathbb{F}_q)$. Equivalently, $C$ is a plane-filling curve $C$ if $\# C(\mathbb{F}_q) = q^2+q+1$. In a natural sense, plane-filling curves are extremal. There are other classes of extremal curves with respect to the set of $\mathbb{F}_q$-points, including blocking curves \cite{AGY23} and tangent-filling curves \cite{AG23}.

From Poonen's work~\cite{Poo04}, we know that there exist smooth plane-filling curves of degree $d$ over $\mathbb{F}_q$ whenever $d$ is sufficiently large with respect to $q$. It is natural to ask for the minimum degree of a smooth plane-filling curve over $\mathbb{F}_q$. Homma and Kim \cite{HK13} proved that the minimum degree is $q+2$.  More precisely, by building on the work of Tallini \cites{Tal61a, Tal61b}, they showed that a plane-filling curve of the form
$$
(ax+by+cz)(x^q y - x y^q) + y(y^q z - y z^q) + z (z^q x - z x^q) = 0 
$$
is smooth if and only if the polynomial $t^3-(c t^2 + b t + a) \in\mathbb{F}_q[t]$ has no $\mathbb{F}_q$-roots. In a sequel paper \cite{Hom20}, Homma investigated further properties of plane-filling curves of degree $q+2$.  The automorphism group of these special curves was studied by Duran Cunha \cite{Dur18}. As another direction, Homma and Kim \cite{HK23} investigated space-filling curves in $\mathbb{P}^{1}\times\mathbb{P}^1$.

In this paper, we investigate the existence of smooth plane-filling curves of degree $q+3$ and higher. The guiding question for our paper is the following. 

\begin{question}\label{quest:q+3}
Let $q$ be a prime power. Does there exist a smooth plane-filling curve of degree $q+3$ defined over $\mathbb{F}_q$?
\end{question}

The three binomials $x^q y - x y^q, y^q z - y z^q$, and $z^q x - z x^q$ generate the ideal of polynomials defining plane-filling curves; see \cite{HK13}*{Proposition 2.1} for proof of this assertion. Thus, any plane-filling curve of degree $q+3$ must necessarily be defined by 
$$
Q_1(x, y, z)  \cdot (x^q y - x y^q) + Q_2(x, y, z)\cdot (y^q z - y z^q) + Q_3(x, y, z) \cdot (z^q x - z x^q)=0
$$
for some homogeneous quadratic polynomials $Q_1, Q_2, Q_3\in\F_q[x,y,z]$. The difficulty is finding suitable $Q_1, Q_2, Q_3$ for which the corresponding curve is smooth.

Our first result gives a necessary and sufficient condition for the plane-filling curve $C_{k}$ to be smooth at all the $\mathbb{F}_q$-points. 

\begin{thm}\label{thm:q+3:Fq:points} For each $k\in\mathbb{F}_q$, consider the plane-filling curve $C_{k}$ defined by
\begin{equation}\label{eq:curve:q+3}
x^2 (x^q y - x y^q) + y^2 (y^q z - y z^q) + (z^2 + k x^2)(z^q x - z x^q) = 0.
\end{equation}
Then $C_k$ is smooth at every $\mathbb{F}_q$-point of $\mathbb{P}^2$ if and only if the polynomial $x^7 + kx^3 -1$ has no zeros in $\mathbb{F}_q$.
\end{thm}

To ensure that the previous theorem is not vacuous, we need to show that there exists some $k\in\mathbb{F}_q$ such that $x^7+k x^3-1$ has no zeros in $\mathbb{F}_q$. 

\begin{prop}\label{prop:existence-good-k} There exists a value $k\in\mathbb{F}_q$ such that $x^7 + kx^3 -1 \in\mathbb{F}_q[x]$ has no zeros in $\mathbb{F}_q$.
\end{prop}

\begin{proof}
When $x=0$, there is no $k\in\Fq$ such that $x^7+kx^3-1=0$. For each $x\in\mathbb{F}_q^{\ast}$, there is a \emph{unique} value of $k\in\Fq$ such that $x^7+kx^3-1=0$. Thus, there are at most $q-1$ values of $k\in\Fq$ such that the polynomial $x^7 + kx^3 - 1$ has a zero in $\Fq$.
\end{proof}

The next result improves Proposition~\ref{prop:existence-good-k}. 

\begin{thm}\label{thm:good-k} There exist at least $\frac{q}{6}-1-\frac{28}{3}\sqrt{q}$ many values of $k\in\mathbb{F}_q$ such that $x^7 + kx^3 -1 \in\mathbb{F}_q[x]$ has no zeros in $\mathbb{F}_q$.
\end{thm}
Note that Theorem~\ref{thm:q+3:Fq:points} and Proposition~\ref{prop:existence-good-k} together yields that for each odd $q$, there exists at least one value $k\in\Fq$ for which the corresponding curve $C_k$ has no singular $\Fq$-points. Furthermore, we expect that the curves in Theorem~\ref{thm:q+3:Fq:points} are smooth if and only if they are smooth at all their $\mathbb{F}_q$-points. Our main conjecture below restates this prediction. 

\begin{conj}\label{main:conjecture:q:odd} Suppose $q$ is odd. The plane-filling curve $C_{k}$ defined by~\eqref{eq:curve:q+3} is smooth if and only if the polynomial $x^7 + kx^3 -1$ has no zeros in $\mathbb{F}_q$.
\end{conj}

We have verified Conjecture~\ref{main:conjecture:q:odd} using Macaulay2 \cite{M2} for all odd prime powers $q<200$. When $q=2^m$ is even, the curve $C_k$ defined by \eqref{eq:curve:q+3} turns out to be singular (for \emph{every} $k\in \mathbb{F}_q$). As a replacement, we consider another curve $D_k$ in this case:
\begin{equation}\label{eq:curve:q+3:even:q}
x^2 (x^q y - x y^q) + y^2 (y^q z - y z^q) + (z^2 + k xy)(z^q x - z x^q) = 0.
\end{equation}

We make a similar conjecture regarding the smoothness of the curves $D_k$.

\begin{conj}\label{main:conjecture:q:even} Suppose $q$ is even. The plane-filling curve $D_{k}$ defined by~\eqref{eq:curve:q+3:even:q} is smooth if and only if the polynomial $x^7 + kx^5 +1 $ has no zeros in $\mathbb{F}_q$.
\end{conj}

The polynomial $x^7+kx^5+1$ featured above is prominent because one can show, similar to Theorem~\ref{thm:q+3:Fq:points}, that a plane-filling curve $D_k$ is smooth at all of its $\Fq$-points (when $q$ is even) if and only if $x^7+kx^5+1$ has no $\Fq$-roots. We have verified Conjecture~\ref{main:conjecture:q:even} using Macaulay2 \cite{M2} for $q=2^m$ when $1\leq m\leq 9$.

We prove the following as partial progress towards Conjecture~\ref{main:conjecture:q:odd}. 

\begin{thm}\label{thm:q+3:Fq^2:points} Suppose $q$ is odd. There exists a suitable choice of $k\in\mathbb{F}_q$ such that the plane-filling curve $C_{k}$ defined by by~\eqref{eq:curve:q+3} is smooth at all $\mathbb{F}_{q^2}$-points.
\end{thm}

A similar argument as the one employed in Theorem 1.7 yields an analogous
result when $q$ is even, and the curve $C_k$ is replaced by $D_k$.

To prove Theorem~\ref{thm:q+3:Fq^2:points}, we will prove that any plane-filling curve of degree $q+3$ which is smooth at $\mathbb{F}_q$-points and has no $\mathbb{F}_q$-linear component must be smooth at each of its $\mathbb{F}_{q^2}$-points. 

We also investigate plane-filling curves of degree $q+r+1$ where $r\geq 2$ is arbitrary. 

\begin{thm}\label{thm:higher-degree} For each $k\in\mathbb{F}_q$, consider the plane-filling curve $C_{k, r}$ defined by
$$
x^r (x^q y - x y^q) + y^r (y^q z - y z^q) + (z^r + k x^r)(z^q x - z x^q) = 0.
$$
Then $C_{k, r}$ is smooth at every $\mathbb{F}_q$-point of $\mathbb{P}^2$ if and only if the polynomial $x^{r^2+r+1}+kx^{r+1}-1=0$ has no zeros in $\mathbb{F}_q$. 
\end{thm}

\subsection*{Structure of the paper} In Section~\ref{sect:q+3}, we prove  Theorem~\ref{thm:good-k}. We devote Section~\ref{sect:Fq^2-points} to Theorem~\ref{thm:q+3:Fq^2:points}, and Section~\ref{sect:higher-degree} to Theorem~\ref{thm:higher-degree}.

\section{Proof of Theorem~\ref{thm:good-k}}\label{sect:q+3}

We begin this section by noting that Theorem~\ref{thm:q+3:Fq:points} is a special case of Theorem~\ref{thm:higher-degree} which will be proven in Section~\ref{sect:higher-degree}. Our Theorem~\ref{thm:q+3:Fq:points} provides a criterion that tests whether the plane-filling curve $C_k$ defined by \eqref{eq:curve:q+3} is smooth at every $\F_q$-point.

The following technical result will be employed in our proof of Theorem~\ref{thm:good-k}.
\begin{lem}\label{lem:irreducible} The polynomial $x^3 y^3 (x+y)(x^2+y^2) + (x^2+xy+y^2)$ is irreducible in $\overline{\mathbb{F}_{q}}[x,y]$.
\end{lem}

\begin{proof} The proof employs a technique seen in Eisenstein's criterion. First, suppose $p=\operatorname{char}(\mathbb{F}_q)\neq 3$. Assume, to the contrary, that $f(x, y)\colonequals x^3 y^3 (x+y)(x^2+y^2) + (x^2+xy+y^2)$ is reducible over the algebraic closure $\overline{\mathbb{F}_{q}}$. Write $f(x, y) = g(x, y)\cdot h(x, y)$, and express
\begin{align*}
g(x, y) &= g_m(x, y) + g_{m+1}(x,y) + \cdots + g_s(x, y) \\  
h(x, y) &= h_n(x, y) +  h_{n+1}(x, y) + \cdots + h_t(x,y) 
\end{align*}
where $g_i(x, y)$ and $h_j(x, y)$ are homogeneous of degree $i$ and $j$, respectively, for $m\leq i\leq s$ and $n\leq j\leq t$. From $f(x, y) = g(x, y) \cdot h(x, y)$, we see that 
$$
\begin{cases}
g_m h_n = x^2 + xy + y^2 \\ 
g_s h_t = x^3 y^3(x+y)(x^2+y^2) \\
\sum_{i+j = k} h_i g_j = 0 \text{ for } 2 < k < 9 
\end{cases}
$$
Since the characteristic $p\neq 3$, the polynomial $x^2+xy+y^2$ factors into distinct linear factors in $\overline{\mathbb{F}_q}[x, y]$. Let $x+\lambda y$ be one of those linear factors with $\lambda\in\overline{\mathbb{F}_q}$. Then $x^2+xy+y^2$ is divisible by $x+\lambda y$ but not by $(x+\lambda y)^2$. Thus, exactly one of $g_m$ or $h_n$ is divisible by $x+\lambda y$. Without loss of generality, assume $x+\lambda y$ divides $g_m$, and not $h_n$. Then using $\sum_{i+j = k} h_i g_j = 0$  for  $2 < k < 9$, we inductively see that $x+\lambda y$ divides $g_j$ for each $m\leq j\leq s$. In particular, $x+\lambda y$ divides $g_s h_t$. This is a contradiction because $x+\lambda y$ does not divide $x^3 y^3(x+y)(x^2+y^2)$. Indeed, $x^2+xy+y^2$ and $x^3 y^3(x+y)(x^2+y^2)$ are relatively prime. 

When $p=3$, a similar argument works from the other end of the polynomial: the leading term $x^3 y^3(x+y)(x^2+y^2)$ is divisible by $x+y$ but not by $(x+y)^2$. We deduce that $f(x, y)$ is irreducible over $\overline{\mathbb{F}_q}$ for every prime power $q$.
\end{proof}

\begin{proof}[Proof of Theorem~\ref{thm:good-k}] Our goal is to give a lower bound on the number of $k\in\mathbb{F}_q$ such that the polynomial $x^7 + kx^3 - 1$ has no roots in $\mathbb{F}_q$. As $x$ ranges in $\mathbb{F}_q^\ast$ (note that there is no $k\in\Fq$ for which $x=0$ would be a root of $x^7+kx^3-1$), the number of ``bad'' choices of $k$ are parametrized by $\frac{1-x^7}{x^3}$. We will show that there are many choices of $x$ and $y$ such that $\frac{1-x^7}{x^3}$ and $\frac{1-y^7}{y^3}$ give rise to the same value of $k$. Setting these expressions equal to each other, we obtain the following.
$$
\frac{1-x^7}{x^3} = \frac{1-y^7}{y^3} \ \ \Rightarrow \ \ 
x^7y^3 - y^3 = y^7x^3 - x^3
$$
After rearranging and dividing both sides by $x-y$, we obtain an affine curve $\mathcal{C}\subset \mathbb{A}^2$ defined by
$$
x^3 y^3 (x+y)(x^2+y^2)+x^2+xy+y^2 = 0,
$$
for $x,y\in\Fq^\ast$ \emph{and} $x\ne y$.  
Let $G$ be a graph whose vertex set is $\mathbb{F}_q^\ast$, and there is an edge between $x$ and $y$ if $(x, y)$ lies on the affine curve $\mathcal{C}$. We consider undirected edges, so the pairs $(x, y)$ and $(y, x)$ correspond to the same edge.

\textbf{Claim 1.} The number of edges of $G$ is at least $\frac{q}{2} -6- 28\sqrt{q}$.

Let $\tilde{\mathcal{C}}\subset\mathbb{P}^2$ be the projectivization of $\mathcal{C}$. By Lemma~\ref{lem:irreducible}, the curve $\tilde{\mathcal{C}}$ is geometrically irreducible. By Hasse-Weil inequality for geometrically irreducible curves \cite{AP96}*{Corollary 2.5}, $\# \tilde{\mathcal{C}}(\mathbb{F}_q) \geq q+1 - 56\sqrt{q}$. Since the line at infinity $z=0$ can contain at most $5$ distinct $\F_q$-points, we have $\# C(\mathbb{F}_q) \geq q - 4-56\sqrt{q}$; furthermore, we exclude the points for which $xy=0$ and there is only one such point $[0:0:1]\in\tilde{\mathcal{C}}$.  We also need to rule out the points on the diagonal, namely $x=y$; in this case, $4x^9+3x^2=0$ which contributes at most $7$ additional points with $x\neq 0$.  Thus, the number of $(x, y)\in C(\mathbb{F}_q)$ with $x\neq y$ is at least $q-12-56\sqrt{q}$. The claim follows since the edges are undirected.

\textbf{Claim 2.} Every connected component of $G$ is a complete graph $K_n$ where $n\in \{1, 2, 3, 4, 5, 6\}$.

If $(x, y)$ and $(x, z)$ are both edges of $G$, then $\frac{1-x^7}{x^3}=\frac{1-y^7}{y^3}$ and $\frac{1-x^7}{x^3}=\frac{1-z^7}{z^3}$. Consequently, $\frac{1-y^7}{y^3}=\frac{1-z^7}{z^3}$ and $(y, z)$ lies on the curve $\mathcal{C}$, so $(y, z)$ is an edge in $G$ too. Thus, each connected component of $G$ is a clique. In addition, from the equation of $\mathcal{C}$, the degree of each vertex $x\in G$ is at most $6$.

For each $1\leq i\leq 6$, let $m_i$ denote the number of cliques of size $i$ in $G$. Counting the number of edges in $G$ leads to the following equality. 
$$
\# E(G) = \sum_{i=1}^{6} \frac{i(i-1)}{2} \cdot m_i.
$$
Each clique of size $i$ in $G$ increases the number of ``good'' values of $k$ by an additive factor of $i-1$ because each clique corresponds to one ``bad'' value of $k$, i.e., a value $k\in\Fq$ for which the equation $x^7+kx^3-1=0$ is solvable for some $x\in\Fq$. More precisely,
\begin{align*}
& \# \{ k \in\mathbb{F}_q \ | \ x^7 + k x^3 - 1 \text{ has no zeros in } \Fq \} \\
& =q - \sum_{i=1}^6 m_i\\
& =1+ (q-1)-\sum_{i=1}^6 m_i\\
& =1+ \sum_{i=1}^6 i\cdot m_i - \sum_{i=1}^6 m_i\\
& = 1+\sum_{i=1}^{6} (i-1) \cdot m_i \\ 
& \geq 1+\frac{1}{3} \sum_{i=1}^{6} \frac{(i-1)i}{2} \cdot m_i \geq 1+\frac{1}{3} \# E(G) \geq 1+\frac{1}{3}\left(\frac{q}{2}-6 - 28\sqrt{q}\right)
\end{align*}
as desired. \end{proof}

\section{Smoothness at $\mathbb{F}_{q^2}$-points} \label{sect:Fq^2-points}

In this section, we show that a plane-filling curve $C$ of degree $q+3$ has the following special property: being smooth at $\mathbb{F}_q$-points implies being smooth at $\mathbb{F}_{q^2}$-points under a mild condition.

\begin{prop}\label{prop:criterion-Fq^2-points} Suppose $C$ is a plane-filling curve of degree $q+3$ such that
\begin{enumerate}[(i)]
    \item \label{condition:smoothness-Fq-points} The curve $C$ is smooth at all the $\mathbb{F}_q$-points.
    \item \label{condition:Fq-linear-components} The curve $C$ has no $\mathbb{F}_q$-linear component.
\end{enumerate}
Then $C$ is smooth at each $\mathbb{F}_{q^2}$-point.
\end{prop}

\begin{proof} Assume, to the contrary, that $C$ is singular at some $\mathbb{F}_{q^2}$-point $Q$. Then $Q$ is not an $\mathbb{F}_q$-point due to the hypothesis~\eqref{condition:smoothness-Fq-points}. Let $Q^{\sigma}$ denote the Galois conjugate of $Q$ under the Frobenius automorphism. More explicitly, if $Q = [x:y:z] \in\mathbb{P}^2$, then $Q^{\sigma} = [x^q:y^q:z^q]$. Note that $Q^{\sigma}$ is also contained in $C$ (since $C$ is defined over $\mathbb{F}_q$). Moreover, $Q^{\sigma}$ is also a singular point of $C$.

Consider the line $L$ joining $Q$ and $Q^{\sigma}$, which is an $\mathbb{F}_q$-line by Galois theory. By hypothesis~\eqref{condition:Fq-linear-components}, the line $L$ must intersect $C$ in exactly $q+3$ points (counted with multiplicity). However, $L$ already contains $q+1$ distinct $\mathbb{F}_q$-points of $C$ (because $C$ is plane-filling), and passes through the two singular points $Q$ and $Q^{\sigma}$, each contributing intersection multiplicity at least $2$. Thus, the total intersection multiplicity between $L$ and $C$ is at least $(q+1) + 2 + 2 = q + 5$, a contradiction. 
\end{proof}

\begin{rem}\label{rem:criterion-Fq^2-points}
We can weaken the hypothesis of Proposition~\ref{prop:criterion-Fq^2-points} by replacing the condition $\deg(C)=q+3$ with $\deg(C)\leq q+4$. Indeed, the same proof works verbatim.
\end{rem}

Next, we show that the plane-filling curves $C_k$ of degree $q+3$ considered in equation ~\eqref{eq:curve:q+3} indeed satisfy condition \eqref{condition:Fq-linear-components} when $q$ is odd.

\begin{prop}\label{prop:Fq-components} The curve $C_k$ defined by ~\eqref{eq:curve:q+3} has no $\mathbb{F}_q$-linear components when $q$ is odd. 
\end{prop}

\begin{proof} There are three types of $\mathbb{F}_q$-lines in $\mathbb{P}^2$.

\textbf{Type I.} The line $L$ is given by $z=0$.

The curve $C_k$ meets the line $\{z=0\}$ at finitely many points determined by $x^2(x^q y - x y^q)=0$. In particular, $\{z=0\}$ is not a component of $C$.

\textbf{Type II.} The line $L$ is given by $x=az$ for some $a\in\mathbb{F}_q$.

The curve $C_k$ meets the line $\{x=az\}$ at finitely many points determined by 
$$
(az)^2 ((az)^q y - (az) y^q) + y^2 (y^q z - y z^q) + (z^2 + k (az)^2)(z^q (az) - z (az)^q) = 0.
$$
After simplifying and using $a^q=a$, the last term cancels and we obtain:
$$
a^3 z^{q+2} y - a^3 z^3 y^q + y^{q+2} z - y^3 z^q = 0 
$$
In particular, $\{x=az\}$ is not a component of $C$.

\textbf{Type III.} The line $L$ is given by $y=ax+bz$ for some $a,b\in\mathbb{F}_q$.

If $a=0$ or $b=0$, then $y=bz$ or $y=ax$, and the analysis is very similar to the previous case. We will assume that $a\neq 0$ and $b\neq 0$. We substitute $y=ax+bz$ into the equation \eqref{eq:curve:q+3} and collect terms to obtain:
\begin{align*}
& (b+a^3-k) x^{q+2} z + (2a^2 b) x^{q+1} z^2 + (b^2 a - 1) x^q z^3 + \\ & (-b - a^3 + k) x^3 z^q + (-2ab) x^2 z^{q+1} + (-ab^2+1) x z^{q+2} = 0 
\end{align*}
The coefficient of $x^{q+1} z^2$ is $2a^2b$, which is nonzero since $q$ is odd (so $2\neq 0$), $a\neq 0$ and $b\neq 0$. Thus, $L$ is not a component of $C_k$. 
\end{proof}

We are now in a position to prove Theorem~\ref{thm:q+3:Fq^2:points} on the existence of $k\in\mathbb{F}_q$ such that the plane-filling curve $C_k$ is smooth at all its $\F_{q^2}$-points.

\begin{proof}[Proof of Theorem~\ref{thm:q+3:Fq^2:points}] The result follows immediately from Proposition~\ref{prop:existence-good-k}, Proposition~\ref{prop:criterion-Fq^2-points}, and Proposition~\ref{prop:Fq-components}. \end{proof}

\section{Higher degree plane-filling curves}\label{sect:higher-degree}

We begin by establishing Theorem~\ref{thm:higher-degree}, which provides a necessary and sufficient condition for the plane-filling curve $C_{k, r}$ to be smooth at all the $\F_q$-points.

\begin{proof}[Proof of Theorem~\ref{thm:higher-degree}]
We consider the curve $C_{k,r}$ given by the equation:
\begin{equation}
\label{eq:1001}
x^r\cdot (x^qy-xy^q) + y^r\cdot (y^qz-yz^q) + (z^r+kx^r)\cdot (z^qx - zx^q)=0.
\end{equation}
We analyze the singular locus of $C_{k,r}$ and get the equations:
\begin{equation}
\label{eq:1002}
rx^{r-1}\cdot (x^qy-xy^q) + x^r\cdot (-y^q) + krx^{r-1}\cdot (z^qx-zx^q)+ (z^r+kx^r)\cdot z^q=0
\end{equation}
\begin{equation}
\label{eq:1003}
x^r\cdot x^q  + ry^{r-1}\cdot (y^qz-yz^q) + y^r\cdot (-z^q) =0
\end{equation}
\begin{equation}
\label{eq:1004}
y^r\cdot y^q + rz^{r-1}\cdot (z^qx -zx^q) + (z^r+kx^r)\cdot (-x^q)=0.
\end{equation}

We next analyze the possibility that we have a singular point when $xyz=0$.

If $x=0$, then equation~\eqref{eq:1002} yields $z=0$, which is then employed in \eqref{eq:1004} to derive $y=0$, contradiction.

If $y=0$, then equation~\eqref{eq:1003} yields $x=0$ and then equation~\eqref{eq:1002} yields $z=0$, contradiction.

If $z=0$, then equation~\eqref{eq:1003} yields $x=0$ and then equation~\eqref{eq:1004} yields $y=0$, contradiction.

So, the only possible singular points are of the form $[x:1:z]$.

We search for possible singular points $[x:1:z]\in\mathbb{P}^2(\Fq)$. Then equations~\eqref{eq:1002},~\eqref{eq:1003}~and~\eqref{eq:1004} read:
\begin{equation}
\label{eq:1005}
-x^r+z^{r+1}+kx^rz=0
\end{equation}
\begin{equation}
\label{eq:1006}
x^{r+1}-z=0
\end{equation}
\begin{equation}
\label{eq:1007}
1-z^rx-kx^{r+1}=0.
\end{equation}
Substituting $z=x^{r+1}$ from equation~\eqref{eq:1006} into equations~\eqref{eq:1005}~and~\eqref{eq:1007}, we obtain
$$-x^r+x^{r^2+2r+1}+kx^{2r+1}=0\text{ and }1-x^{r^2+r+1}-kx^{r+1}=0,$$
that is, there exists a singular $\Fq$-rational point on $C_{k,r}$ if and only if there exists $x\in\Fq^\ast$ such that
\begin{equation}
\label{eq:1008}
x^{r^2+r+1}+kx^{r+1}-1=0,
\end{equation}
as desired. 

We end the proof by mentioning that some care is needed to treat the case when the characteristic $p$ of the field divides the degree of the curve (i.e., $p$ divides $r+1$ in this setting). Indeed, the singular locus of any projective curve $\{f=0\}$ is defined by $\{f=\frac{\partial f}{\partial x} = \frac{\partial f}{\partial y} =\frac{\partial f}{\partial z} =0\}$. When $p$ divides $\deg(f)$, it is \emph{not} enough to consider the points in the locus $\{\frac{\partial f}{\partial x} =\frac{\partial f}{\partial y} =\frac{\partial f}{\partial z} =0\}$. Fortunately, in our case, the $\Fq$-point $[x:1:z]$ is automatically on the curve $C_{k,r}$ because $C_{k,r}$ is plane-filling.
\end{proof} 

It may be natural to make a prediction identical to Conjecture~\ref{main:conjecture:q:odd} for higher-degree curves. However, some care is needed, as the following two examples show. We found these examples using Macaulay2~\cite{M2}.

\begin{ex} 
Let $r=5$, $q=11$, and $k=9$. The plane-filling curve $C_{9, 5}$ over $\F_{11}$ is smooth at all the $\F_{11}$-points because the polynomial $x^{31}+9x^{6}-1$ is an irreducible polynomial over $\F_{11}$. However, $C_{9, 5}$ is singular at two Galois-conjugate $\F_{11^{2}}$-points. 
\end{ex}

In the previous example, the curve $C_{9, 5}$ is irreducible over $\F_{11}$. Thus, $C_{9, 5}$ satisfies the two conditions of Theorem~\ref{prop:criterion-Fq^2-points} and yet it is singular at two $\mathbb{F}_{11^2}$-points. Since $\deg(C_{9, 5})=q+6$, we see that Remark~\ref{rem:criterion-Fq^2-points} is close to being sharp.

\begin{ex} 
Let $r=7$, $q=5$. In this case, the plane-filling curve $C_{k, 7}$ defined over $\F_{5}$ is singular for each $k\in \F_5$. Indeed, the associated polynomial $x^{57}+kx^{8}-1$ has an $\F_{5}$-root for $k\in \{0, 2, 3, 4\}$. For these values of $k$, the curve $C_{k, r}$ is singular at an $\F_5$-point. For $k=1$, the curve $C_{1, 7}$ is singular at four points, namely, two pairs of Galois-conjugate $\F_{5^{2}}$-points. 
\end{ex}

The two examples above illustrate that Conjecture~\ref{main:conjecture:q:odd} needs to be modified for plane-filling curves of degree $q+r+1$ when $r$ is arbitrary. We propose two related conjectures on the smoothness of the curve $C_{k, r}$ from Theorem~\ref{thm:higher-degree}. Recall that $C_{k, r}\subset \mathbb{P}^2$ is defined by
\begin{equation*}
    x^r (x^qy-xy^q) + y^r (y^qz-yz^q) + (z^r+kx^r) (z^qx - zx^q)=0
\end{equation*}
where $r\geq 2$ is a positive integer and $k\in\Fq$. 

\begin{conj}\label{main:conjecture:higher-degree-1} Let $r\geq 2$. There exists an integer $m\colonequals m(r)$ with the following property. For all finite fields $\Fq$ with cardinality $q>m$ and characteristic not dividing $r$, there exists some $k\in\Fq$ such that the curve $C_{k, r}$ is smooth.
\end{conj}

Using Macaulay2~\cite{M2}, we enumerated through values of $r$ in the range $[2, 17]$ and $q$ in the range $[2, 100]$ with $\gcd(r, q)=1$. We found only the following pairs $(r, q)$ for which $C_{k, r}$ is singular for \emph{every} $k\in\Fq$: $(r, q) = (7, 5)$, $(13, 3)$, $(16, 9)$, and $(17, 7)$.

\begin{conj}\label{main:conjecture:higher-degree-2} Let $r\geq 2$. There exists an integer $s\colonequals s(r)$ with the following property. For all finite fields $\Fq$ with characteristic not dividing $r$, and for all $k\in \Fq$, if $C_{k, r}$ is smooth at all of its $\mathbb{F}_{q^{s}}$-points, then $C_{k, r}$ is smooth. 
\end{conj}

As a motivation for Conjecture~\ref{main:conjecture:higher-degree-2}, we mention the following general fact about pencils of plane curves. The family of plane curves $C_k$ forms a \emph{pencil} of plane curves since the parameter $k\in\mathbb{F}_q$ appears linearly in the defining equation. If $\mathcal{L}$ is a pencil of plane curves in $\mathbb{P}^2$ parametrized by $\mathbb{A}^1$, then $\Fq$-members of $\mathcal{L}$ are defined by $f(x, y, z) + k g(x, y, z)=0$ where $k\in\Fq$ is arbitrary. We will use $X_k$ to denote this plane curve in the following proposition.

\begin{prop}\label{prop:checking-finite-level}
Let $\mathcal{L}$ be a pencil of plane curves $\{X_k\}_{k\in\Fq}$ of degree $d$ defined over a finite field $\mathbb{F}_q$. Suppose that for every $s\geq 1$, there exists some $k\in\Fq$ such that $X_k$ is smooth at all of its $\mathbb{F}_{q^s}$-points. Then there exists some $\ell\in\Fq$ such that $X_{\ell}$ is smooth. 
\end{prop}

\begin{proof}
Assume, to the contrary, that $X_k$ is singular for each $k\in \Fq$. 
For each $k\in\Fq$, let $n_k\in\mathbb{N}$ such that the curve $X_k$ is singular at some $\F_{q^{n_k}}$-point. Let $N\colonequals \prod_{k\in\Fq} n_k$. By construction, no $X_k$ is smooth at all of its $\F_{q^N}$-points, contradicting the hypothesis. 
\end{proof}

Proposition~\ref{prop:checking-finite-level} asserts that to find a smooth member of any pencil $\mathcal{L}$ defined over $\Fq$, it is sufficient to find a member which is smooth at all points of an (arbitrary) finite degree. Conjecture~\ref{main:conjecture:higher-degree-2} strengthens the conclusion by predicting that for a pencil of plane-filling curves, one finds a smooth member by only checking smoothness at all points of \emph{fixed} finite degree.

\end{document}